\documentclass[11pt]{article}

\usepackage[letterpaper,margin=1.05in]{geometry}
\usepackage{amsmath,amssymb,amsthm,mathtools}
\usepackage{microtype}
\usepackage{enumitem}
\usepackage[colorlinks=true,linkcolor=blue,citecolor=blue,urlcolor=blue]{hyperref}

\newtheorem{theorem}{Theorem}[section]
\newtheorem{proposition}[theorem]{Proposition}
\newtheorem{lemma}[theorem]{Lemma}
\newtheorem{corollary}[theorem]{Corollary}
\newtheorem{assumption}[theorem]{Assumption}
\theoremstyle{definition}

\theoremstyle{remark}
\newtheorem{remark}[theorem]{Remark}

\newcommand{\cN}{\mathcal N}

\newcommand{\cT}{\mathcal T}

\newcommand{\bP}{\mathbb P}
\newcommand{\bQ}{\mathbb Q}
\newcommand{\bE}{\mathbb E}
\newcommand{\one}{\mathbf 1}
\newcommand{\dd}{\,\mathrm d}
\newcommand{\ip}[2]{\langle #1,#2\rangle}

\title{Dynamical Equations for Poisson Galton--Watson Trees
and Component Densities of Sparse Inhomogeneous Random Graphs}

\author{Bohan Hu\\\texttt{hbh995d8@mail.ustc.edu.cn}
  \and Wen Sun\\\texttt{wensun.ustc@gmail.com}}
\date{}

\hypersetup{
  pdftitle={Dynamical Equations for Poisson Galton--Watson Trees and Component Densities of Sparse Inhomogeneous Random Graphs},
  pdfauthor={Bohan Hu and Wen Sun},
  pdfsubject={Branching processes, pruning, and inhomogeneous random graphs},
  pdfkeywords={branching process, change of measure, inhomogeneous random graph, pruning, spinal decomposition}
}

\begin{document}
\maketitle

\begin{center}
\small
School of Mathematical Sciences, University of Science and Technology of China,\\
Jinzhai Road 96, Hefei 230026, China.\\[2mm]
\end{center}

\begin{abstract}
We study Poisson Galton--Watson trees on a standard Borel type space when the
offspring kernel is multiplied by a scalar parameter.  On finite trees, we
identify the Radon--Nikodym derivative between two parameter values and show
that it remains measurable after projection to the total progeny measure.
Under a uniform bound on the offspring intensities, differentiation yields
exact differential and integral equations for the projected laws without
irreducibility, reversibility, or a positive eigenfunction.  With an additional
positive eigenfunction bounded above and away from zero, we relate these
equations to an infinite spinal tree, uniform pruning, the Doob transform, and
the Aldous--Pitman ascension process.  For a uniformly bounded offspring kernel,
we also prove
uniform exponential integrability of the total progeny throughout the
spectrally subcritical regime.  As an application, under the graphical-kernel
assumptions of Bollob\'as, Janson and Riordan, the
number $K_n$ of connected components satisfies
\[
 \frac{K_n}{n}\longrightarrow \bE_\pi\!\left[\frac1{T_u}\right]
\]
in probability and in $L^1$, where $T_u$ is the total progeny of the associated
branching process and $1/\infty=0$.  If $q_u(x)$ is its extinction probability
from type $x$, re-rooting and extinction duality give the explicit limit
\[
 \int_S q_u\,\dd\pi-\frac{u}{2}
 \iint_{S^2}\kappa(x,y)q_u(x)q_u(y)\,\pi(\dd x)\pi(\dd y).
\]
This extends the finite-type and compact-continuous formulas to the full BJR
graphical-kernel setting, allowing separable noncompact type spaces and kernels
that may be unbounded or reducible.
\end{abstract}



\section{Introduction}

The local exploration of a sparse inhomogeneous random graph is naturally
approximated by a Poisson Galton--Watson tree whose particles carry types.
In the homogeneous Erd\H{o}s--R\'enyi model, tree-component counts and their
Gaussian fluctuations were studied by Pittel \cite{Pittel1990}, while
Puhalskii \cite{Puhalskii2005} analysed normal, moderate, and large deviations
for component-count processes.  Finite-type inhomogeneous extensions were
developed by S\"oderberg \cite{Soderberg2002}.  Under the more general
graphical-kernel conditions of Bollob\'as, Janson and Riordan, Theorem~9.1 of
\cite{BJR2007} gives the asymptotic proportion of vertices in components of
each fixed order; see also the branching-process construction in Section~2 of
that paper.  This suggests studying graph observables through functionals of
the associated branching process.

In this paper we vary the intensity of every offspring process by a scalar
parameter $u$.  When $0\leq u\leq1$, this operation can be coupled by retaining
each edge of the tree at parameter $1$ independently with probability $u$.  If
$\Lambda_u$ denotes the resulting total progeny measure and $T_u=\Lambda_u(S)$
its total mass, then the law of $\Lambda_u$ satisfies an exact evolution equation.
The key observation is simple: on every finite family tree $t$, changing $u$
changes the Radon--Nikodym derivative only through the number of edges
$|t|-1$ and the sum of the offspring intensities
$\sum_{z\in t}m(x_z)$ over all vertices.  Those two quantities are measurable
functions of the total progeny measure $\Lambda_t$.  Consequently, the
projection from trees to total progeny loses no information relevant to this
change of measure.
The term \emph{dynamical} in the title refers to evolution in this intensity
parameter; the resulting measure-valued equation is not asserted to be the
forward equation of a time-homogeneous Markov process.

The central result, Theorem~\ref{thm:measure-RN}, is the Radon--Nikodym identity
\begin{equation}\label{eq:intro-RN}
 \frac{\dd P_{\alpha,u}^{\Lambda}}{\dd P_{\alpha,v}^{\Lambda}}(\mu)
 =\left(\frac uv\right)^{\mu(S)-1}
   \exp\{- (u-v)\ip{\mu}{M\one}\},\qquad u,v>0,
\end{equation}
on the space of finite integer-valued measures.  Here $M$ is the unscaled
offspring kernel and $\alpha$ is the root-type law.  Differentiating
\eqref{eq:intro-RN} yields the localized and integrable forms of the governing
differential and integral equations in Theorem~\ref{thm:dynamics}.  The
change-of-measure identity is the result in the paper that holds at the greatest
level of generality.  Unlike a spine or ascension argument, its proof requires
neither Perron--Frobenius theory nor a change of measure on infinite trees.

Under the additional assumption that $M\varphi=\rho\varphi$ for a function
bounded above and away
from zero, a complementary tree-valued description is available: we construct
the $\varphi$-size-biased infinite tree and show that off-spine thinning at a
lower intensity maps it to the corresponding size-biased tree at that
intensity (Proposition~\ref{prop:pruning-spine}).  This is a structural
complement to the general Radon--Nikodym identity.  The direct proofs of
Theorems~\ref{thm:measure-RN} and~\ref{thm:dynamics} are independent of
Section~\ref{sec:spine}; under the eigenfunction assumption, however, the
spinal construction gives a second, structural derivation of the dynamical
equations.  The Doob $\varphi$-transform normalizes the
total offspring intensity to one.  Its untyped genealogy is therefore the
Poisson Galton--Watson pruning process studied by Aldous and Pitman: their
transition and ascension results are in Sections~4.2--4.5 of
\cite{AldousPitman1998}.  We lift their ascension representation to the typed
tree and show that its one-time mixture identity is exactly the normalized form
of our integral dynamical equation.  We then transport this identity back
through the $\varphi$-transform and recover the equations for the original
kernel.  The role of Section~\ref{sec:spine} is to record the Poisson-kernel
typed forms of these classical results, prove the pruning commutation, and
establish their compatibility with the general dynamics.

The size-biased construction in Section~2, particularly equations
(2.1)--(2.2), of Lyons, Pemantle and Peres
\cite{LyonsPemantlePeres1995} is a classical source for spinal changes of
measure.  Section~12, particularly Proposition~12.1, of Biggins and Kyprianou
\cite{BigginsKyprianou2004} develops the corresponding martingale change of
measure for Galton--Watson processes on a general type space.  The present
spinal construction specializes that framework to a Poisson offspring kernel
and records its compatibility with uniform edge pruning.  Beyond the
Aldous--Pitman Poisson edge-pruning process, Abraham, Delmas and He
\cite{AbrahamDelmasHe2012} study node pruning for general Galton--Watson trees
and give its ascension laws and path representation in Propositions~4.6--4.7,
while Theorem~3 of He and Winkel
\cite{HeWinkel2014} gives an invariance principle for edge-pruning processes.

Theorem~\ref{thm:exp-full} gives a short resolvent proof of exponential
integrability of $T_u$ whenever
$u\rho(M)<1$, where $\rho(M)$ is the spectral radius of $M$ on bounded
measurable functions.  Thus the conclusion holds throughout this spectrally
subcritical regime, even when the maximal one-step mean is larger than one.

Finally, let $K_n$ be the number of connected components in the BJR
inhomogeneous random graph.  Theorem~9.1 of \cite{BJR2007} gives the asymptotic
proportion of vertices in components of every fixed order.  Summing these limits
with a uniform tail bound yields
$K_n/n\to\bE_\pi[1/T_u]$ in Theorem~\ref{thm:BJR-component}, where $T_u$ is
the total progeny of the associated branching process and $1/\infty=0$.
Detailed balance, extinction duality, and the re-rooting identity of
Section~\ref{sec:rerooting} then give the explicit formula in
Theorem~\ref{thm:IRG-general}.

Related formulas have been established under different assumptions and within
different frameworks.  In finite type, Corollary~1 of Yu and Sun
\cite{YuSun2024} gives the component-count law of large numbers away from
criticality.  Their Remark~1 states the reciprocal-total-progeny representation
and suggests that it should admit a proof based on pruning.  Andreis, K\"onig,
Langhammer and Patterson assume a compact metric type space and a continuous
irreducible kernel.  Their Theorem~2.1 and equation~(2.5)
\cite{AndreisEtAl2023} give a law of large numbers for the microscopic
component measure, and their equation~(2.12) directly relates that measure to
the total-progeny law of the associated branching process.  Lemmas~6.7--6.8 of
that paper identify its total mass, including at criticality.  Thus the
branching-process interpretation is already present under their assumptions.
Our component-density theorem instead allows the separable, possibly
noncompact type spaces and the almost-everywhere continuous, integrable,
possibly unbounded or reducible graphical kernels of \cite{BJR2007}.  This
greater generality concerns the law of large numbers for the component density;
it does not subsume the large- or moderate-deviation principles in those two
papers.  The method here is also different: it combines the BJR fixed-component
asymptotics with a reversible-tree identity and branching-process duality.

\paragraph{Scope of the assumptions.}
The finite-tree and total-progeny Radon--Nikodym identities
(Proposition~\ref{prop:tree-RN} and Theorem~\ref{thm:measure-RN}) require only
$M(x,S)<\infty$ at every type, which ensures that the Poisson offspring process
is well defined there; they do not require the uniform bound
\eqref{eq:bounded-kernel}.  The dynamical equations
(Theorem~\ref{thm:dynamics}) are stated under that uniform bound, which controls
$A(\mu)$ by $D\mu(S)$ both on finite strata and in the global integrability
argument.  Section~\ref{sec:spine} additionally assumes the bounded positive
eigenfunction in Assumption~\ref{ass:eigenfunction}.  The exponential-moment
result assumes \eqref{eq:bounded-kernel} and $u\rho(M)<1$ on $B_b(S)$.  The
re-rooting identity uses instead detailed balance and almost-sure finiteness,
with finite average degree for its negative-first-moment corollary.  Finally,
Section~\ref{sec:graphs} states the BJR graphical-kernel assumptions and permits
kernels that are not uniformly bounded.  Thus no eigenfunction, reversibility,
or random-graph hypothesis is implicit in the general change-of-measure
argument.

The paper is organized as follows.  Section~\ref{sec:model} defines the tree and
the pruning coupling.  Section~\ref{sec:spine} develops the spinal,
$\varphi$-transform, and ascension structure under an additional eigenfunction
hypothesis.  Section~\ref{sec:dynamics} proves the Radon--Nikodym and
dynamical identities without that hypothesis.  Section~\ref{sec:exp} treats exponential moments.
Section~\ref{sec:rerooting} proves the re-rooting identity under detailed balance,
and Section~\ref{sec:graphs} derives the random-graph component-density limits.

\section{Poisson Galton--Watson trees}\label{sec:model}

\subsection{Type space and offspring kernel}

Throughout, $(S,\mathcal S)$ is a standard Borel space.  Let $M$ be a finite
kernel from $S$ to $S$: for every $x$, $M(x,\cdot)$ is a finite measure, and for
every $A\in\mathcal S$, the map $x\mapsto M(x,A)$ is measurable.  Write
\[
 m(x):=M(x,S),\qquad
 Mf(x):=\int_S f(y)M(x,\dd y).
\]
Unless stated otherwise, we assume
\begin{equation}\label{eq:bounded-kernel}
 D:=\sup_{x\in S}m(x)<\infty.
\end{equation}
Thus $M$ is a bounded positive operator on the Banach lattice $B_b(S)$ of
bounded measurable functions with the supremum norm.

Let $\mathbb U=\{\varnothing\}\cup\bigcup_{k\geq1}\mathbb N^k$ be the
Ulam--Harris tree, with root $\varnothing$.  A typed family tree is a rooted
subtree $t\subset\mathbb U$ together with a type $x_v\in S$ for every $v\in t$.
The usual cylinder sigma-field makes the collection of typed locally finite
trees a standard Borel space.  Proposition~6.2 of \cite{LastPenrose2018}
provides a measurable enumeration of finite point configurations on a Borel
space; after identifying $S$ with a Borel subset of $[0,1]$, fix such an
enumeration for every offspring configuration.  This embeds the genealogy in
$\mathbb U$.  None of the laws or formulas below depends on the chosen
enumeration.

Denote this measurable tree space by $\mathfrak T$, and let
$\mathfrak T_f:=\{t\in\mathfrak T:|t|<\infty\}$ with the trace sigma-field.
Both the finiteness event and the map from $\mathfrak T_f$ to the space of
finite integer-valued measures on $S$ that sends a finite tree to its total
progeny measure are measurable.  Indeed,
\[
 |t|=\sum_{v\in\mathbb U}\one_{\{v\in t\}},\qquad
 \Lambda_t(A)=\sum_{v\in\mathbb U}
   \one_{\{v\in t,\ x_v\in A\}},\quad A\in\mathcal S,
\]
are increasing limits of measurable functions.  Hence
$\mathfrak T_f=\{|t|<\infty\}$ is measurable; since the evaluation maps
$\mu\mapsto\mu(A)$ generate the sigma-field on the point-measure space, so is
$t\mapsto\Lambda_t$ on $\mathfrak T_f$.  We write
$\bP_{\alpha,u}^{\mathrm{fin}}$ for the
restriction of the tree law to $\mathfrak T_f$; it is in general a
subprobability measure.  Recording the zero-offspring configuration at every
leaf makes elements of $\mathfrak T_f$ complete family trees rather than
finite truncations of possibly infinite trees.

For $u\geq0$, the Poisson Galton--Watson tree $G_u$ with kernel $uM$ is defined
as follows.  Conditional on a vertex having type $x$, the point measure of the
types of its children is a Poisson point process with intensity $uM(x,\dd y)$;
different vertices reproduce independently.  We write $\bP_{x,u}$ when the root
has fixed type $x$, and $\bP_{\alpha,u}=\int\bP_{x,u}\alpha(\dd x)$ when its type
has law $\alpha$.

Let $Z_{u,k}$ be the point measure of the types in generation $k$.  The total
progeny measure and total progeny are
\[
 \Lambda_u:=\sum_{k\geq0}Z_{u,k},\qquad T_u:=\Lambda_u(S),
\]
where $T_u$ may be infinite.  Let $\cN_f(S)$ denote the standard Borel space of
finite integer-valued measures on $S$.  The restriction of the law of
$\Lambda_u$ to $\cN_f(S)$ is denoted by $P_{\alpha,u}^{\Lambda}$.  It is a
probability measure when the tree becomes extinct almost surely and otherwise is
a subprobability measure.

\subsection{Uniform pruning}

For $0\leq u\leq v$, a tree with kernel $vM$ can be thinned to one with kernel
$uM$: independently retain each parent--child edge with probability $u/v$ and
keep the component of the root.  Corollary~5.9 of \cite{LastPenrose2018} (the
thinning theorem) shows that the result has law $G_u$.  Tree-valued pruning
chains in the Galton--Watson setting are developed in Sections~3.2--3.4 of Aldous and Pitman
\cite{AldousPitman1998}.  Node pruning and the associated ascension process for
general offspring laws are developed by Abraham, Delmas and He
\cite[Propositions~4.6--4.7]{AbrahamDelmasHe2012}; scaling limits for
branch-point (node) pruning and edge pruning are studied by He and Winkel
\cite{HeWinkel2014}, whose Theorems~2 and~3 are the respective invariance
principles.

In particular, attach independent uniform random variables $U_e$ to the edges of
$G_1$, and retain an edge exactly when
\begin{equation}\label{eq:correct-pruning}
 U_e\leq u.
\end{equation}
The root component has law $G_u$.  Notice that \eqref{eq:correct-pruning}
retains each edge with probability $u$, so $G_1$ is unchanged and $G_0$ consists
only of the root.

The proofs below only use the marginal laws.  They therefore also apply for
$u>1$, where the family may instead be coupled using a Poisson process in the
intensity parameter.  This latter coupling is an intensity-growth coupling,
not edge-retention pruning, because a retention probability cannot exceed one.
Thus every statement for $u>1$ concerns the genuine tree with kernel $uM$;
only its realization by retaining edges of $G_1$ is unavailable.

\section{Spinal bias, the \texorpdfstring{$\varphi$}{phi}-transform, and ascension}
\label{sec:spine}

The results in this section are structural properties of the tree-valued pruning
process.  Unlike the Radon--Nikodym argument in
Section~\ref{sec:dynamics}, they require an additional eigenfunction.  Keeping
the two arguments separate makes clear which conclusions hold for a general
bounded kernel and which belong to the Aldous--Pitman spinal picture.  The
direct Radon--Nikodym, dynamical, and exponential-moment proofs in
Sections~\ref{sec:dynamics} and~\ref{sec:exp} do not rely on this section.
After obtaining the normalized Aldous--Pitman mixture identity, we nevertheless
pull it back through the $\varphi$-transform to give a second derivation of the
dynamical equations under Assumption~\ref{ass:eigenfunction}.

\begin{assumption}[Bounded positive eigenfunction]\label{ass:eigenfunction}
There are $\rho>0$ and $\varphi\in B_b(S)$ such that
\begin{equation}\label{eq:eigenfunction}
 M\varphi=\rho\varphi,
 \qquad 0<a\leq\varphi(x)\leq b<\infty\quad(x\in S).
\end{equation}
\end{assumption}

\subsection{The additive martingale and the infinite spine}

The size-biased tree construction is classical in the one-type case; see
Lyons, Pemantle and Peres \cite{LyonsPemantlePeres1995}.  The general-type-space
martingale change of measure is treated by Biggins and Kyprianou
\cite{BigginsKyprianou2004}.  We give the Poisson-kernel specialization needed
for the pruning identities below.

Let $\mathcal F_n$ be the sigma-field generated by the typed tree through
generation $n$.  For $u>0$ define
\begin{equation}\label{eq:additive-martingale}
 W_{u,n}:=(u\rho)^{-n}\ip{Z_{u,n}}{\varphi}.
\end{equation}

\begin{proposition}[Spinal change of measure]\label{prop:spinal-change}
Under Assumption~\ref{ass:eigenfunction}, $(W_{u,n})_{n\geq0}$ is a
nonnegative $\bP_{x,u}$-martingale with mean $\varphi(x)$.  The consistent
finite-generation densities
\begin{equation}\label{eq:Q-density}
 \left.\frac{\dd\bQ_{x,u}}{\dd\bP_{x,u}}\right|_{\mathcal F_n}
 =\frac{W_{u,n}}{\varphi(x)}
\end{equation}
define a probability law $\bQ_{x,u}$ on trees.  Moreover, there is a joint law
of a tree and a distinguished infinite ray
$\xi=(\xi_0,\xi_1,\ldots)$ whose tree marginal is $\bQ_{x,u}$.

More explicitly, set
\begin{equation}\label{eq:spine-transition}
 \widehat M(x,\dd y)
 :=\frac{\varphi(y)}{\rho\varphi(x)}M(x,\dd y).
\end{equation}
At a spine vertex of type $x$, take an ordinary Poisson point process with
intensity $uM(x,\dd y)$ and add one distinguished child with type law
$\widehat M(x,\dd y)$; that child continues the spine.  Every nonspine vertex
reproduces with the original Poisson intensity $uM$.  This construction has
tree marginal $\bQ_{x,u}$.
\end{proposition}

\begin{proof}
Conditional on $\mathcal F_n$, the branching property and
\eqref{eq:eigenfunction} give
\begin{align*}
 \bE_{x,u}[W_{u,n+1}\mid\mathcal F_n]
 &=(u\rho)^{-(n+1)}
   \sum_{z:\,|z|=n}uM\varphi(x_z)\\
 &=(u\rho)^{-n}\sum_{z:\,|z|=n}\varphi(x_z)
 =W_{u,n}.
\end{align*}
Also $W_{u,0}=\varphi(x)$, proving the first claim and the consistency of
\eqref{eq:Q-density}.  More explicitly, the martingale property says that the
measure with density $W_{u,n+1}/\varphi(x)$ restricts on $\mathcal F_n$ to the
one with density $W_{u,n}/\varphi(x)$.  The finite-generation tree spaces are
standard Borel and their restriction maps form a projective system; the
Kolmogorov extension theorem therefore gives the tree law $\bQ_{x,u}$.

It remains to identify the local construction.  Let $\eta$ be a Poisson point
process with intensity $uM(x,\dd y)$.  The offspring law at a spine vertex is
biased by
\begin{equation}\label{eq:size-biased-offspring}
 \frac{\ip{\eta}{\varphi}}{u\rho\varphi(x)}.
\end{equation}
After this bias, choose the next spine child from $\eta$ proportionally to its
$\varphi$-weight.  For every nonnegative measurable $F$ of a point
configuration and a distinguished point, the Mecke equation gives
\begin{align}
 &\frac1{u\rho\varphi(x)}
   \bE\int_S\varphi(y)F(\eta,y)\eta(\dd y)\notag\\
 &\hspace{2cm}=
   \int_S\widehat M(x,\dd y)\,
          \bE F(\eta+\delta_y,y).
 \label{eq:Mecke-spine}
\end{align}
Thus the biased configuration is an ordinary configuration plus the
distinguished child in \eqref{eq:spine-transition}.  The Mecke equation used
here is Theorem~4.1 of \cite{LastPenrose2018}.  Iterating
\eqref{eq:Mecke-spine} along the distinguished ray gives exactly the density
\eqref{eq:Q-density} at every generation.  The same consistent finite-level
construction, now retaining the distinguished vertices, gives the joint law of
the tree and its infinite ray, and proves the remaining assertions.
\end{proof}

In the one-type critical or subcritical setting, this is the classical tree
conditioned to survive and its spinal decomposition; see Proposition~2 and
Corollary~3 of \cite{AldousPitman1998}.  In a general type space we use the
unambiguous term \emph{$\varphi$-size-biased tree}: identifying it with a limit
conditioned on survival may require additional irreducibility or ratio-limit
hypotheses, none of which are needed here.

\subsection{Pruning and spinal bias}

Write $G_{x,u}^{\infty}$ for the infinite spinal tree of
Proposition~\ref{prop:spinal-change}.  Notice that its spine transition
$\widehat M$ does not depend on $u$.

\begin{proposition}[Pruning commutes with spinal bias]
\label{prop:pruning-spine}
Let $0<u\leq v$.  Start with $G_{x,v}^{\infty}$, retain every off-spine edge
independently with probability $u/v$, retain every spine edge, and delete the
descendants separated by the removed edges.  The resulting infinite spinal tree
has law $G_{x,u}^{\infty}$.
\end{proposition}

\begin{proof}
By definition of the off-spine pruning protocol in the statement, every spine
edge is retained deterministically.  Thus only ordinary offspring and their
descendant bushes are thinned.
At a spine vertex, Proposition~\ref{prop:spinal-change} decomposes the offspring
as an ordinary Poisson process with intensity $vM$ plus the distinguished spine
child with transition $\widehat M$.  Independent thinning changes the first
intensity to $uM$ and leaves the distinguished child unchanged.  The same
thinning turns each nonspine descendant tree with kernel $vM$ into one with
kernel $uM$.  Independence is preserved by Corollary~5.9 of
\cite{LastPenrose2018}, so the local rules are precisely those of
$G_{x,u}^{\infty}$.
\end{proof}

There is also a useful formula when the spine itself is pruned.  Start with
$G_{x,1}^{\infty}$, retain every edge independently with probability $u<1$,
and let $G_{x,u}^{*}$ be the root component.  Denote its law by
$\bP_{x,u}^{*}$ and the generation of a vertex $z$ in a tree $t$ by $|z|$.

\begin{corollary}[Terminal-spine bias]\label{cor:terminal-spine}
If $0\leq u<1$ and $u\rho<1$, then $G_{x,u}^{*}$ and $G_u$ are finite almost surely and
\begin{equation}\label{eq:terminal-spine-density}
 \frac{\dd\bP_{x,u}^{*}}{\dd\bP_{x,u}}(t)
 =H_{x,u}^{*}(t)
 :=\frac{1-u}{\varphi(x)}
   \sum_{z\in t}\rho^{-|z|}\varphi(x_z).
\end{equation}
When $\rho=1$, the bias is $(1-u)\ip{\Lambda_t}{\varphi}/\varphi(x)$.
\end{corollary}

\begin{proof}
The case $u=0$ is immediate, so suppose $u>0$.
First prune only the off-spine edges.  By
Proposition~\ref{prop:pruning-spine}, the result is $G_{x,u}^{\infty}$.
The number $H$ of consecutive retained spine edges is independent and satisfies
$\bP(H=h)=(1-u)u^h$.  Repeated application of
\eqref{eq:Mecke-spine} shows, as an equality of measures on finite typed trees
with a marked vertex $z$ at height $h$, that
\begin{equation}\label{eq:terminal-marked-measure}
 \bP(G_{x,u}^{*}\in\dd t,\,\xi_H=z,\,H=h)
 =\frac{1-u}{\varphi(x)}\rho^{-h}\varphi(x_z)
   \bP_{x,u}(G_u\in\dd t).
\end{equation}
Indeed, the $h$ retained spine edges contribute $u^h$, whereas the $h$ spinal
changes of measure contribute $(u\rho)^{-h}\varphi(x_z)/\varphi(x)$.
Summing \eqref{eq:terminal-marked-measure} over all $z\in t$ gives
\eqref{eq:terminal-spine-density}.

Finally, positivity of $M$, the bounds on $\varphi$, and
$M\varphi=\rho\varphi$ give
\[
 M^n\one\leq a^{-1}M^n\varphi
 =a^{-1}\rho^n\varphi\leq (b/a)\rho^n\one.
\]
Since $u\rho<1$, it follows that
\[
 \sup_x\bE_{x,u}T_u
 =\sup_x\sum_{n\geq0}u^nM^n\one(x)
 \leq \frac{b/a}{1-u\rho}<\infty,
\]
so $G_u$ is finite almost surely.  Formula
\eqref{eq:terminal-spine-density} integrates to one because
\begin{align*}
 \bE_{x,u}H_{x,u}^{*}(G_u)
 &=\frac{1-u}{\varphi(x)}
   \sum_{n\geq0}\rho^{-n}\bE_{x,u}\ip{Z_{u,n}}{\varphi}\\
 &=(1-u)\sum_{n\geq0}u^n=1.
\end{align*}
Moreover, the pruned spine has finite geometric length and only finitely many
off-spine bushes are attached to it; each bush is finite by the preceding bound.
Thus $G_{x,u}^{*}$ is finite almost surely as well.
\end{proof}

\subsection{The \texorpdfstring{$\varphi$}{phi}-transform}

The kernel $\widehat M$ in \eqref{eq:spine-transition} is a Markov kernel because
$\widehat M(x,S)=1$.  On bounded functions,
\begin{equation}\label{eq:operator-similarity}
 \widehat M=\rho^{-1}D_{\varphi}^{-1}MD_{\varphi},
\end{equation}
where $D_\varphi f=\varphi f$.  Since $D_\varphi$ and its inverse are bounded,
$\widehat M$ has spectral radius one.  Consequently the eigenvalue $\rho$ in
Assumption~\ref{ass:eigenfunction} equals the spectral radius $\rho(M)$ of $M$
on $B_b(S)$.  Let $\widehat G_u$ be the Poisson
Galton--Watson tree with offspring kernel $u\widehat M$ and fixed root type $x$;
write its law as $\widehat\bP_{x,u}$.

\begin{proposition}[The $\varphi$-transform]\label{prop:phi-transform}
For every $u>0$, on finite typed trees rooted at type $x$,
\begin{equation}\label{eq:phi-transform-density}
 \frac{\dd\widehat\bP_{x,u}^{\mathrm{fin}}}
      {\dd\bP_{x,u/\rho}^{\mathrm{fin}}}(t)
 =\frac1{\varphi(x)}
   \exp\left\{u\sum_{z\in t}\left(\frac{m(x_z)}\rho-1\right)\right\}
   \prod_{z\in t}\varphi(x_z)^{1-N_z},
\end{equation}
where $N_z$ is the number of children of $z$.  In particular, if $\rho=1$ the
two trees in \eqref{eq:phi-transform-density} have the same parameter $u$.
\end{proposition}

\begin{proof}
At a vertex $z$, comparison of the Poisson Janossy measures gives the exponential
factor
\(
 \exp\{u(m(x_z)/\rho-1)\}
\)
and, for every edge $(z,w)$, the kernel ratio
$\varphi(x_w)/\varphi(x_z)$.  Multiplying over all vertices and edges yields
\[
 \prod_{(z,w)\in E(t)}\frac{\varphi(x_w)}{\varphi(x_z)}
 =\frac1{\varphi(x)}\prod_{z\in t}\varphi(x_z)^{1-N_z}.
\]
The finite-height induction used in Proposition~\ref{prop:tree-RN} applies
verbatim and proves the equality of measures on all finite trees.  This avoids
point probabilities of exact offspring configurations, which are generally zero
on a continuous type space.
\end{proof}

\subsection{The normalized tree and its ascension process}

In this subsection we write $\lambda$ for the scalar intensity parameter, in
accord with the notation of Aldous and Pitman.
Because $\widehat M(x,S)=1$, the untyped genealogy of $\widehat G_\lambda$ is
the ordinary one-type $\operatorname{PGW}(\lambda)$ tree.  Conditional on that
genealogy and the root type, types propagate along its edges with transition
kernel $\widehat M$, independently for different children given their parent
types.  We use independent atomless auxiliary marks to order siblings and then
forget those marks, so this decoration is independent of the genealogical
ordering.  Thus the Poisson pruning process of
\cite{AldousPitman1998} has a canonical typed lift.

Let $(\widehat G_\lambda,\lambda\geq0)$ denote this increasing tree-valued
Markov process: it has the above marginal at time $\lambda$, and pruning
$\widehat G_\mu$ with retention probability $\lambda/\mu$ gives
$\widehat G_\lambda$.  Its existence and its forward attachment construction
are equation~(69), Proposition~20, and equation~(76) of
\cite{AldousPitman1998}.  In the typed lift, an attachment at a vertex of type
$x$ occurs at total rate one, the new child has law $\widehat M(x,\dd y)$, and
the tree below that child has law $\widehat G_\lambda$ started at $y$.
Write $\bP_x$ for the law of this process with root type $x$.

Write $\widehat T_\lambda$ for the total progeny of $\widehat G_\lambda$ and
define the ascension time and extinction probability by
\begin{equation}\label{eq:ascension-def}
 A:=\inf\{\lambda:\widehat T_\lambda=\infty\},
 \qquad q_\lambda:=\bP_x(\widehat T_\lambda<\infty).
\end{equation}
Since $\widehat M\one=\one$, every particle has
$\operatorname{Poisson}(\lambda)$ total offspring, independently of its type.
Thus the untyped genealogy, and hence its extinction probability $q_\lambda$,
does not depend on $x$.  It is the smallest solution in $[0,1]$ of
\begin{equation}\label{eq:scalar-extinction}
 q_\lambda=e^{-\lambda(1-q_\lambda)}.
\end{equation}
For $\lambda>1$, put $\widehat\lambda=\lambda q_\lambda<1$.

\begin{proposition}[Laws at ascension]\label{prop:ascension-laws}
For $\lambda\geq1$,
\begin{equation}\label{eq:A-distribution}
 \bP_x(A\leq\lambda)=1-q_\lambda,
\end{equation}
and, for $\lambda>1$,
\begin{align}
 \bP_x(\widehat G_{A-}\in\dd t,A\in\dd\lambda)
 &=(1-q_\lambda)T(t)\widehat\bP_{x,\lambda}(\dd t)\dd\lambda,
 \label{eq:A-joint}\\
 \bP_x(\widehat G_{A-}\in\dd t\mid A=\lambda)
 &=(1-\widehat\lambda)T(t)
   \widehat\bP_{x,\widehat\lambda}(\dd t).
 \label{eq:A-conditional}
\end{align}
If $U$ is uniform on $(0,1)$, then
\begin{equation}\label{eq:A-uniform}
 (A,q_A)\stackrel d=
 \left(\frac{-\log U}{1-U},U\right).
\end{equation}
These are the typed versions of Lemma~22, equations~(81)--(84), of
\cite{AldousPitman1998}.
\end{proposition}

\begin{proof}
Monotonicity of the process gives
$\{A\leq\lambda\}=\{\widehat T_\lambda=\infty\}$, proving
\eqref{eq:A-distribution}.  Equation~\eqref{eq:scalar-extinction} follows from
the Poisson generating function and $\widehat M\one=\one$.

The attachment rate at each vertex is one by equation~(76) of
\cite{AldousPitman1998}.  An attached tree is infinite with probability
$1-q_\lambda$, independently of the type of the attachment vertex.  Hence a
finite tree $t$ jumps to the infinite state at rate $T(t)(1-q_\lambda)$; this is
equation~(77) of that paper and yields \eqref{eq:A-joint}.

We spell out the extinction duality at the point where the one-type Poisson
structure of the normalized genealogy is used.  Let $\eta$ be the offspring
point process of a particle of type $x$, so that
$\eta\sim\Pi_{\lambda\widehat M(x,\cdot)}$.  Conditional on $\eta$, extinction
of all descendant bushes has probability $q_\lambda^{\eta(S)}$, since
$q_\lambda$ is independent of the child type.  Thus conditioning on extinction
tilts the offspring law by $q_\lambda^{\eta(S)}$.  Since
$\widehat M(x,S)=1$ and $q_\lambda=e^{-\lambda(1-q_\lambda)}$,
\[
 q_\lambda^{\eta(S)}
 \Pi_{\lambda\widehat M(x,\cdot)}(\dd\eta)
 =q_\lambda\,
 \Pi_{\lambda q_\lambda\widehat M(x,\cdot)}(\dd\eta).
\]
After division by the normalizing probability $q_\lambda$, the conditional
offspring process is Poisson with intensity
$\lambda q_\lambda\widehat M=\widehat\lambda\widehat M$.  Applying the same
argument recursively to every descendant shows that the entire tree
conditioned on extinction has law $\widehat G_{\widehat\lambda}$.  Consequently,
on finite trees,
\begin{equation}\label{eq:normalized-duality}
 \widehat\bP_{x,\lambda}^{\mathrm{fin}}(\dd t)
 =q_\lambda\widehat\bP_{x,\widehat\lambda}(\dd t).
\end{equation}
Differentiating \eqref{eq:scalar-extinction} gives
\(
 -q_\lambda'=q_\lambda(1-q_\lambda)/(1-\widehat\lambda)
\).
Divide \eqref{eq:A-joint} by this density and use
\eqref{eq:normalized-duality} to obtain \eqref{eq:A-conditional}.
Finally, inversion of $1-q_\lambda$ gives \eqref{eq:A-uniform}, exactly as in
Lemma~22 of \cite{AldousPitman1998}.
\end{proof}

Let $\widehat G_1^\infty$ be the spinal tree for the Markov kernel
$\widehat M$, and let $(\widehat G_s^*,0\leq s\leq1)$ be its uniform pruning,
including pruning of the spine.  Corollary~\ref{cor:terminal-spine} with
$\rho=1$ and $\varphi=1$ gives
\begin{equation}\label{eq:normalized-size-bias}
 \widehat\bP_{x,s}^{*}(\dd t)
 =(1-s)T(t)\widehat\bP_{x,s}(\dd t),\qquad 0\leq s<1.
\end{equation}
This is the typed extension of Corollary~24, equation~(86), of
\cite{AldousPitman1998}.

\begin{proposition}[Aldous--Pitman representation]\label{prop:AP-representation}
Let $U$ be uniform on $(0,1)$ and independent of
$(\widehat G_s^*,0\leq s\leq1)$.  Then
\begin{equation}\label{eq:AP-path-representation}
 (\widehat G_\lambda,0\leq\lambda<A)
 \stackrel d=
 \left(\widehat G_{\lambda U}^{*},
       0\leq\lambda<\frac{-\log U}{1-U}\right).
\end{equation}
Consequently, for every fixed $0<u\leq1$,
\begin{equation}\label{eq:AP-one-time}
 \widehat G_u\stackrel d=\widehat G_{uU}^{*}.
\end{equation}
\end{proposition}

\begin{proof}
After types are forgotten, \eqref{eq:AP-path-representation} is Proposition~26,
equation~(91), of \cite{AldousPitman1998}; its one-time consequence is their
equation~(92).  Conditional on every finite untyped genealogy, the type law is
obtained by putting the fixed root type $x$ at the root and applying
$\widehat M$ independently along its directed edges.  This conditional marking
law is independent of the parameter and is preserved when edges are pruned.
It therefore gives the same conditional typed law on both sides of the untyped
process identity, proving the typed lift.
\end{proof}

Combining \eqref{eq:AP-one-time} with
\eqref{eq:normalized-size-bias} gives, for every measurable set $B$ of finite
typed trees,
\begin{equation}\label{eq:AP-integral-identity}
 \widehat\bP_{x,u}(B)
 =\frac1u\int_0^u\widehat\bP_{x,s}^{*}(B)\dd s
 =\frac1u\int_0^u\int_B(1-s)T(t)
       \widehat\bP_{x,s}(\dd t)\dd s.
\end{equation}
Since $\widehat M\one=\one$, equation \eqref{eq:AP-integral-identity} is the
normalized form of the integral dynamical equation.  We next transport it back
to the original kernel.  This gives a second derivation of the governing
equations; Section~\ref{sec:dynamics} gives the direct proof without
Assumption~\ref{ass:eigenfunction}.

\subsection{Pulling the normalized dynamics back to the original kernel}
\label{subsec:pullback}

For a finite typed tree $t$, write
\[
 T(t):=|t|,
 \qquad
 A(t):=\sum_{z\in t}m(x_z).
\]
Proposition~\ref{prop:phi-transform}, applied at parameter $\rho r$, gives
\begin{equation}\label{eq:phi-transform-rho-r}
 \widehat\bP_{x,\rho r}^{\mathrm{fin}}(\dd t)
 =H_{x,r}(t)\,\bP_{x,r}^{\mathrm{fin}}(\dd t),
\end{equation}
where
\begin{equation}\label{eq:pullback-H}
 H_{x,r}(t)
 :=C_x(t)\exp\{r(A(t)-\rho T(t))\},
 \qquad
 C_x(t):=\frac1{\varphi(x)}
          \prod_{z\in t}\varphi(x_z)^{1-N_z}.
\end{equation}
The factor $C_x(t)$ is independent of the intensity parameter $r$.

\begin{proposition}[Pull-back of the Aldous--Pitman identity]
\label{prop:pullback-AP}
Under Assumption~\ref{ass:eigenfunction}, for $0<u\leq1/\rho$ the finite-tree
laws satisfy
\begin{equation}\label{eq:pullback-tree-volterra}
 \bP_{x,u}^{\mathrm{fin}}(\dd t)
 =\frac1u\int_0^u
 T(t)(1-\rho s)
 \exp\{(s-u)(A(t)-\rho T(t))\}
 \bP_{x,s}^{\mathrm{fin}}(\dd t)\dd s.
\end{equation}
Consequently, for every root law $\alpha$, the finite total-progeny laws obey
\begin{equation}\label{eq:pullback-measure-volterra}
 P_{\alpha,u}^{\Lambda}(\dd\mu)
 =\frac1u\int_0^u T(\mu)(1-\rho s)
 \exp\{(s-u)(A(\mu)-\rho T(\mu))\}
 P_{\alpha,s}^{\Lambda}(\dd\mu)\dd s,
\end{equation}
where $T(\mu)=\mu(S)$ and $A(\mu)=\ip{\mu}{M\one}$.
\end{proposition}

\begin{proof}
Apply \eqref{eq:AP-integral-identity} at the normalized parameter
$\rho u\leq1$:
\[
 \widehat\bP_{x,\rho u}^{\mathrm{fin}}(\dd t)
 =\frac1{\rho u}\int_0^{\rho u}
 (1-r)T(t)\widehat\bP_{x,r}^{\mathrm{fin}}(\dd t)\dd r.
\]
With $r=\rho s$, this becomes
\[
 \widehat\bP_{x,\rho u}^{\mathrm{fin}}(\dd t)
 =\frac1u\int_0^u
 (1-\rho s)T(t)\widehat\bP_{x,\rho s}^{\mathrm{fin}}(\dd t)\dd s.
\]
Substitute \eqref{eq:phi-transform-rho-r} at parameters $u$ and $s$.  The
factor $C_x(t)$ cancels, and
\[
 \frac{H_{x,s}(t)}{H_{x,u}(t)}
 =\exp\{(s-u)(A(t)-\rho T(t))\},
\]
which proves \eqref{eq:pullback-tree-volterra}.  All remaining factors depend
on $t$ only through $\Lambda_t$.  Projection under
$t\mapsto\Lambda_t$ therefore proves \eqref{eq:pullback-measure-volterra} for
a fixed root type; integration with respect to $\alpha$ proves it for an
arbitrary root law.
\end{proof}

\begin{corollary}[Spinal derivation of the governing equations]
\label{cor:spinal-dynamics}
Under Assumption~\ref{ass:eigenfunction}, let
$F:\cN_f(S)\to\mathbb R$ be bounded and supported on $\cN_N(S)$ for some
$N$.  For $0<u<1/\rho$,
\begin{equation}\label{eq:spinal-weak-ODE}
 \frac{\dd}{\dd u}\int F(\mu)P_{\alpha,u}^{\Lambda}(\dd\mu)
 =\int F(\mu)\left\{\frac{T(\mu)-1}{u}-A(\mu)\right\}
 P_{\alpha,u}^{\Lambda}(\dd\mu),
\end{equation}
and
\begin{equation}\label{eq:spinal-weak-integral}
 \int F(\mu)P_{\alpha,u}^{\Lambda}(\dd\mu)
 =\frac1u\int_0^u\int
 F(\mu)\{T(\mu)-sA(\mu)\}
 P_{\alpha,s}^{\Lambda}(\dd\mu)\dd s.
\end{equation}
Thus the $\varphi$-transform and ascension construction recover the same
dynamical equations as Theorem~\ref{thm:dynamics} on the spectrally
subcritical side.
\end{corollary}

\begin{proof}
Put
\[
 J_F(u):=\int F(\mu)P_{\alpha,u}^{\Lambda}(\dd\mu),
 \qquad
 B(\mu):=A(\mu)-\rho T(\mu).
\]
Multiplying \eqref{eq:pullback-measure-volterra} by $F(\mu)$ and integrating
gives
\begin{equation}\label{eq:pullback-J}
 uJ_F(u)
 =\int_0^u\int
 F(\mu)T(\mu)(1-\rho s)e^{(s-u)B(\mu)}
 P_{\alpha,s}^{\Lambda}(\dd\mu)\dd s.
\end{equation}
Because $F$ is supported on $\cN_N(S)$ and
$A(\mu)\leq DT(\mu)\leq DN$ there, the integrand and its derivative in $u$
are uniformly bounded when $u$ ranges over a compact subinterval of
$(0,1/\rho)$.  Leibniz' rule therefore applies and gives
\begin{align*}
 J_F(u)+uJ_F'(u)
 &=\int F(\mu)T(\mu)(1-\rho u)
       P_{\alpha,u}^{\Lambda}(\dd\mu)\\
 &\quad-\int_0^u\int
 F(\mu)T(\mu)(1-\rho s)B(\mu)e^{(s-u)B(\mu)}
 P_{\alpha,s}^{\Lambda}(\dd\mu)\dd s.
\end{align*}
Applying \eqref{eq:pullback-measure-volterra} with the bounded test function
$F(\mu)B(\mu)$ identifies the second integral as
\[
 u\int F(\mu)B(\mu)P_{\alpha,u}^{\Lambda}(\dd\mu).
\]
It follows that
\begin{align*}
 J_F(u)+uJ_F'(u)
 &=\int F(\mu)
 \{T(\mu)(1-\rho u)-u(A(\mu)-\rho T(\mu))\}
 P_{\alpha,u}^{\Lambda}(\dd\mu)\\
 &=\int F(\mu)\{T(\mu)-uA(\mu)\}
 P_{\alpha,u}^{\Lambda}(\dd\mu),
\end{align*}
which is \eqref{eq:spinal-weak-ODE}.  Equivalently,
\[
 \frac{\dd}{\dd u}\{uJ_F(u)\}
 =\int F(\mu)\{T(\mu)-uA(\mu)\}
 P_{\alpha,u}^{\Lambda}(\dd\mu).
\]
Since $|uJ_F(u)|\leq u\|F\|_\infty$, integration from zero to $u$ proves
\eqref{eq:spinal-weak-integral}.
\end{proof}

\begin{remark}
The restriction $u\leq1/\rho$ in Proposition~\ref{prop:pullback-AP} comes only
from the range $\rho u\leq1$ of the Aldous--Pitman one-time identity.  The
direct Radon--Nikodym argument in Section~\ref{sec:dynamics} removes both this
restriction and Assumption~\ref{ass:eigenfunction}: locally on finite-mass
strata, the same dynamical equations hold for every $u>0$.
\end{remark}

\section{Change of measure and governing equations}\label{sec:dynamics}

This section gives the main Radon--Nikodym argument directly; it does not use
Assumption~\ref{ass:eigenfunction} or any result from Section~\ref{sec:spine}.
The Poisson nature of the offspring distribution makes the change of parameter
particularly transparent.  Recall that if $\Pi_{u\xi}$ and $\Pi_{v\xi}$ are the
laws of Poisson point processes with finite intensity measures $u\xi$ and
$v\xi$, then, for $u,v>0$,
\begin{equation}\label{eq:PPP-RN}
 \frac{\dd\Pi_{u\xi}}{\dd\Pi_{v\xi}}(\eta)
 =\exp\{-(u-v)\xi(S)\}\left(\frac uv\right)^{\eta(S)}.
\end{equation}
Indeed, on the stratum $\{\eta(S)=k\}$ the Janossy measure is
$e^{-u\xi(S)}u^k\xi^{\otimes k}/k!$; comparison with the corresponding measure
at $v$ proves \eqref{eq:PPP-RN}.  This is also an immediate consequence of the
Poisson Janossy formula in Example~4.8, equation~(4.21), of
\cite{LastPenrose2018}.

\begin{proposition}[Finite-tree Radon--Nikodym derivative]\label{prop:tree-RN}
For $u,v>0$, the finite-tree subprobability measures defined in
Section~\ref{sec:model} satisfy
\begin{equation}\label{eq:tree-RN}
 \frac{\dd\bP_{\alpha,u}^{\mathrm{fin}}}
      {\dd\bP_{\alpha,v}^{\mathrm{fin}}}(t)
 =\left(\frac uv\right)^{|t|-1}
   \exp\left\{-(u-v)\sum_{z\in t}m(x_z)\right\}.
\end{equation}
The identity holds for every root law $\alpha$.
\end{proposition}

\begin{proof}
Let $\mathsf T_h$ be the measurable stratum of complete trees of height at most
$h$; ``complete'' means that the zero-offspring event is recorded for every
leaf.  We prove the result for the restrictions of the two tree measures to
$\mathsf T_h$ by induction on $h$.  For $h=0$, the root has no children, and
\eqref{eq:PPP-RN} gives the factor $e^{-(u-v)m(x_\varnothing)}$.

For the induction step, first expose the root's finite offspring point measure.
Its Radon--Nikodym derivative is
\[
 e^{-(u-v)m(x_\varnothing)}(u/v)^{N_\varnothing}.
\]
Conditional on the offspring types, the descendant trees are independent, and
their restrictions to $\mathsf T_{h-1}$ have the Radon--Nikodym derivatives supplied by
the induction hypothesis.  Multiplication gives
\[
 \prod_{z\in t}e^{-(u-v)m(x_z)}(u/v)^{N_z}.
\]
This is an induction for finite kernels of measures, so it proves the
Radon--Nikodym identity on the whole stratum, not merely on cylinder events.
Since $\sum_{z\in t}N_z=|t|-1$, the displayed product equals
\eqref{eq:tree-RN}.  Finally, the finite-tree space is the increasing union
$\bigcup_{h\geq0}\mathsf T_h$; applying the stratum identity to
$B\cap\mathsf T_h$ and then using monotone convergence proves it on every
measurable set $B$ of finite trees.  The common root law contributes no factor.
\end{proof}

For $\mu\in\cN_f(S)$ define
\[
 T(\mu):=\mu(S),\qquad A(\mu):=\ip{\mu}{M\one}
 =\int_Sm(x)\mu(\dd x).
\]

\begin{theorem}[Total-progeny Radon--Nikodym derivative]\label{thm:measure-RN}
For every $u,v>0$,
\begin{equation}\label{eq:measure-RN}
 \frac{\dd P_{\alpha,u}^{\Lambda}}
      {\dd P_{\alpha,v}^{\Lambda}}(\mu)
 =\left(\frac uv\right)^{T(\mu)-1}
   e^{-(u-v)A(\mu)},
 \qquad P_{\alpha,v}^{\Lambda}\text{-a.e. }\mu.
\end{equation}
In particular, all finite-total-progeny laws with positive parameters are
mutually absolutely continuous.
\end{theorem}

\begin{proof}
The Radon--Nikodym derivative in Proposition~\ref{prop:tree-RN} is a measurable function
of the total progeny measure, because
\[
 |t|=\Lambda_t(S),\qquad
 \sum_{z\in t}m(x_z)=\ip{\Lambda_t}{M\one}.
\]
Write the right-hand side of \eqref{eq:measure-RN} as $R_{u,v}(\mu)$.  For every
measurable $B\subseteq\cN_f(S)$, Proposition~\ref{prop:tree-RN} gives directly
\begin{align*}
 \int_B R_{u,v}(\mu)P_{\alpha,v}^{\Lambda}(\dd\mu)
 &=\bE_{\alpha,v}[\one_{\{\Lambda_v\in B,\,T_v<\infty\}}
                    R_{u,v}(\Lambda_v)]\\
 &=\bP_{\alpha,u}(\Lambda_u\in B,\,T_u<\infty)
 =P_{\alpha,u}^{\Lambda}(B).
\end{align*}
This is precisely \eqref{eq:measure-RN} and also proves mutual absolute
continuity.
\end{proof}

The derivative of a measure-valued map must be interpreted with an integrability
condition.  For $N\geq1$, put
\[
 \cN_N(S):=\{\mu\in\cN_f(S):T(\mu)\leq N\}.
\]
On this set, both $T$ and $A\leq DT$ are bounded.

\begin{theorem}[Dynamical equations]\label{thm:dynamics}
Let $F:\cN_f(S)\to\mathbb R$ be bounded and supported on $\cN_N(S)$ for some
$N$.  Then $u\mapsto\int F\dd P_{\alpha,u}^{\Lambda}$ is continuously
differentiable on $(0,\infty)$, and
\begin{equation}\label{eq:weak-ODE}
 \frac{\dd}{\dd u}\int F(\mu)P_{\alpha,u}^{\Lambda}(\dd\mu)
 =\int F(\mu)\left\{\frac{T(\mu)-1}{u}-A(\mu)\right\}
     P_{\alpha,u}^{\Lambda}(\dd\mu).
\end{equation}
Moreover,
\begin{equation}\label{eq:weak-integral}
 \int F(\mu)P_{\alpha,u}^{\Lambda}(\dd\mu)
 =\frac1u\int_0^u\int F(\mu){T(\mu)-sA(\mu)}
       P_{\alpha,s}^{\Lambda}(\dd\mu)\dd s.
\end{equation}
Equivalently, on every $\cN_N(S)$, in the sense of finite signed measures,
\begin{align}
 \frac{\dd}{\dd u}P_{\alpha,u}^{\Lambda}(\dd\mu)
 &=\left\{\frac{T(\mu)-1}{u}-A(\mu)\right\}
    P_{\alpha,u}^{\Lambda}(\dd\mu),\label{eq:measure-ODE}\\
 P_{\alpha,u}^{\Lambda}(\dd\mu)
 &=\frac1u\int_0^u\{T(\mu)-sA(\mu)\}
    P_{\alpha,s}^{\Lambda}(\dd\mu)\dd s.\label{eq:measure-integral}
\end{align}
Let $I\subset(0,\infty)$ be a compact interval on which
$\sup_{s\in I}\bE_{\alpha,s}T_s<\infty$.  For every bounded $F$, the map
$u\mapsto u\int F\dd P_{\alpha,u}^{\Lambda}$ is absolutely continuous on $I$,
and \eqref{eq:weak-ODE} holds for Lebesgue-a.e. $u\in I$.  If the same moment
bound holds on $[0,u]$, then \eqref{eq:weak-integral} also holds for that $u$.
Here $T_s$ is an extended nonnegative random variable; in particular, the
moment assumptions in this paragraph imply $T_s<\infty$ almost surely at the
parameters under consideration.
\end{theorem}

\begin{proof}
Fix $v>0$ and insert \eqref{eq:measure-RN} into the left-hand side of
\eqref{eq:weak-ODE}.  On $\cN_N(S)$ the density and its derivative are bounded
uniformly when $u$ ranges over a compact subset of $(0,\infty)$.  Differentiation
under the integral is therefore justified and gives \eqref{eq:weak-ODE}.

Multiplying \eqref{eq:weak-ODE} by $u$ gives
\begin{equation}\label{eq:derivative-uP}
 \frac{\dd}{\dd u}\left(u\int F\dd P_{\alpha,u}^{\Lambda}\right)
 =\int F(\mu)\{T(\mu)-uA(\mu)\}
        P_{\alpha,u}^{\Lambda}(\dd\mu).
\end{equation}
The term on the left converges to zero as $u\downarrow0$, since
\[
 \left|u\int F\dd P_{\alpha,u}^{\Lambda}\right|
 \leq u\|F\|_\infty\longrightarrow0
\]
and $P_{\alpha,u}^{\Lambda}$ is a subprobability measure.  Integrating
\eqref{eq:derivative-uP} from $0$ to $u$ proves
\eqref{eq:weak-integral}.  The measure formulations follow by testing against
bounded measurable functions supported on $\cN_N(S)$.

For the final assertion, apply \eqref{eq:derivative-uP} to
$F_N=F\one_{\{T\leq N\}}$ and integrate between two endpoints of $I$.  Since
\[
 |F_N(\Lambda_s)(T_s-sA(\Lambda_s))|
 \leq \|F\|_\infty(1+sD)T_s,
\]
the assumed uniform first-moment bound and dominated convergence in
$\dd s\,\dd\bP$ permit $N\to\infty$.  The resulting integral representation
shows that $u\mapsto u\int F\dd P_{\alpha,u}^{\Lambda}$ is absolutely
continuous.  Since $u$ is bounded away from zero on $I$, division by $u$ and
the product rule recover \eqref{eq:weak-ODE} for Lebesgue-a.e. $u\in I$.
If the moment bound holds down to zero, the same argument on
$[\varepsilon,u]$, followed by $\varepsilon\downarrow0$, gives
\eqref{eq:weak-integral}.
\end{proof}

\begin{remark}
The local formulation in Theorem~\ref{thm:dynamics} remains valid at critical
parameters, where $\bE T_u$ may be infinite.  Writing the equation globally
without a test-function or integrability qualification can otherwise lead to an
undefined difference of two infinite quantities.
\end{remark}

\section{Exponential moments in the spectrally subcritical regime}\label{sec:exp}

Let $\rho(M)$ denote the spectral radius of $M$ as a bounded operator on
$B_b(S)$.  The elementary domination $m(x)\leq D$ gives exponential moments
when $uD<1$.  The next theorem replaces this sufficient condition by the weaker
spectral condition $u\rho(M)<1$ on $B_b(S)$.

\begin{theorem}[Subcritical exponential integrability]\label{thm:exp-full}
Assume \eqref{eq:bounded-kernel} and let $u\geq0$ satisfy
\begin{equation}\label{eq:strict-subcritical}
 u\rho(M)<1.
\end{equation}
Then there exists $\theta>0$ such that
\begin{equation}\label{eq:uniform-exp}
 \sup_{x\in S}\bE_{x,u}e^{\theta T_u}<\infty.
\end{equation}
More precisely, for every compact set
$J\subset\{s\geq0:s\rho(M)<1\}$ there are $\theta_J>0$ and $C_J<\infty$
such that
\[
 \sup_{s\in J}\sup_{x\in S}\bE_{x,s}e^{\theta_JT_s}\leq C_J.
\]
\end{theorem}

\begin{proof}
The case $u=0$ is immediate, so assume $u>0$.  Since the spectral radius of
$uM$ is smaller than one, the positive resolvent
\[
 h:=(I-uM)^{-1}\one=\sum_{k\geq0}(uM)^k\one
\]
converges in $B_b(S)$.  Indeed, the spectral-radius formula gives an integer
$n_0$ such that $\|(uM)^{n_0}\|<1$.  Grouping the Neumann series by residues
modulo $n_0$ reduces it to a finite sum of operator-norm convergent geometric
series; positivity of $M$ makes the resulting resolvent positive.  Put
$H=\|h\|_\infty$.  Then
\begin{equation}\label{eq:resolvent-h}
 1\leq h\leq H,\qquad uMh=h-1.
\end{equation}

For a nonnegative measurable $f$, the Poisson generating functional gives the
map
\[
 \Psi_\theta(f)(x)
 :=\exp\{\theta+uM(f-1)(x)\}.
\]
Starting from $f_0=1$ and setting $f_{n+1}=\Psi_\theta(f_n)$, the function $f_n$
has an exact truncated-tree interpretation.  Namely, let
\[
 T_u^{[n]}:=\sum_{k=0}^{n}Z_{u,k}(S).
\]
Then $f_1(x)=e^\theta=\bE_{x,u}e^{\theta T_u^{[0]}}$.  If $\eta$ is the
offspring point process of the root, conditioning on $\eta$ gives inductively
\begin{align*}
 \bE_{x,u}e^{\theta T_u^{[n+1]}}
 &=e^\theta\bE\left[\prod_{y\in\eta}
      \bE_{y,u}e^{\theta T_u^{[n]}}\right]\\
 &=\exp\{\theta+uM(f_{n+1}-1)(x)\}
 =f_{n+2}(x).
\end{align*}
Thus $f_{n+1}(x)=\bE_{x,u}e^{\theta T_u^{[n]}}$ exactly.  Since
$T_u^{[n]}\uparrow T_u$, monotone convergence yields
$f_n(x)\uparrow\bE_{x,u}e^{\theta T_u}$, even before finiteness of the limit is
known.

Choose $a>0$ so small that
\begin{equation}\label{eq:choose-a}
 aH\leq1,\qquad \frac e2aH^2\leq\frac12,
\end{equation}
and put $\theta=a/2$.  We claim that $g=e^{ah}$ is a supersolution of
$\Psi_\theta$.  For $0\leq z\leq1$,
$e^z-1\leq z+(e/2)z^2$; indeed,
$e^z-1-z=\int_0^z(z-t)e^t\dd t\leq(e/2)z^2$ on this interval.
Hence, using $h^2\leq Hh$ and
\eqref{eq:resolvent-h},
\begin{align*}
 \theta+uM(g-1)
 &\leq \theta+a uMh+\frac e2a^2uM(h^2)\\
 &\leq \theta+a(h-1)+\frac e2a^2H(h-1)\\
 &\leq ah+\theta-a+\frac e2a^2H^2
 \leq ah.
\end{align*}
Thus $\Psi_\theta(g)\leq g$.  Monotonicity of $\Psi_\theta$ gives
$f_n\leq g\leq e^{aH}$ for all $n$.  Monotone convergence now proves
\eqref{eq:uniform-exp}.

For the uniform assertion, let $u_0=\max J$.  Then
$u_0\rho(M)<1$.  Use
$h_0=(I-u_0M)^{-1}\one$.  Since $uMh_0\leq u_0Mh_0=h_0-1$, the same
supersolution and the same bound work for every $u\in J$.
\end{proof}

An explicit, although sometimes more restrictive, bound follows from scalar
domination.

\begin{proposition}[Explicit Poisson domination]\label{prop:scalar-domination}
If $\lambda=uD<1$, then $T_u$ is stochastically dominated, uniformly in the root
type, by the total progeny $T^{(\lambda)}$ of a one-type Galton--Watson process
with $\operatorname{Poisson}(\lambda)$ offspring.  Consequently,
\begin{equation}\label{eq:explicit-theta}
 \sup_x\bE_{x,u}e^{\theta T_u}<\infty
 \quad\text{for every}\quad
 0\leq\theta<\lambda-1-\log\lambda.
\end{equation}
For $\lambda=0$ all exponential moments are finite.
\end{proposition}

\begin{proof}
Given a particle of type $x$, its total number of children is
$\operatorname{Poisson}(um(x))$.  Couple it below an independent
$\operatorname{Poisson}(uD)$ variable by adding a Poisson variable of mean
$u(D-m(x))$.  Iterating this coupling over the family tree proves the stochastic
 domination.  If $T_n^{(\lambda)}$ is the progeny through generation $n$, its
 moment generating function is obtained by iterating
$F_0=e^\theta$ and $F_{n+1}=e^\theta e^{\lambda(F_n-1)}$.  Thus the moment
generating function of $T^{(\lambda)}$ is the monotone limit of this iteration
and is the minimal solution in $[1,\infty]$ of
\[
 F=e^\theta e^{\lambda(F-1)}.
\]
Equivalently, a finite solution satisfies
\[
 \theta=\phi(F):=\log F-\lambda(F-1).
\]
For $0<\lambda<1$, the concave function $\phi$ has its maximum at
$F=1/\lambda$, and
$\phi(1/\lambda)=\lambda-1-\log\lambda$.  If $\theta$ is below this maximum,
the smaller solution bounds every iterate $F_n$, so monotone convergence gives a
finite moment.  This proves \eqref{eq:explicit-theta}; the case $\lambda=0$ is
immediate.
\end{proof}

\section{Detailed balance and re-rooting}\label{sec:rerooting}

Let $K$ be another finite kernel on $S$ and let $\pi$ be a probability measure.
We say that $(\pi,K)$ satisfies detailed balance if the measure
\begin{equation}\label{eq:detailed-balance}
 \pi(\dd x)K(x,\dd y)
\end{equation}
on $S^2$ is symmetric.  Let $\cT$ be a Poisson Galton--Watson tree with root law
$\pi$ and offspring kernel $K$.  This re-rooting symmetry is closely related to
the mass-transport principle in Definition~2.1 of \cite{AldousLyons2007}.

\begin{lemma}[Re-rooting invariance]\label{lem:rerooting}
Suppose \eqref{eq:detailed-balance} holds and $\cT$ is finite almost surely.
Conditional on the unrooted typed tree underlying $\cT$, the root is uniformly
distributed over its vertices.
\end{lemma}

\begin{proof}
We give the measure argument in detail because pointwise ``densities of types''
need not exist on a general standard Borel space.  Conditional on the finite
tree, attach to every vertex an independent auxiliary mark with the atomless
uniform law on $[0,1]$.  Almost surely the marks are distinct, so their ranks
label the vertices of every finite realization by $[n]$.  These marks are
independent of the types and will be forgotten at the end.

Fix an unrooted labelled abstract tree $\tau$ on $[n]$ and a proposed root
$r\in[n]$, and orient every edge away from $r$.  Let $Q_{\tau,r}$ be the finite
measure on $S^n$ obtained by restricting the marked branching-tree law to the
event that the ranked unrooted tree is $\tau$ and its root has label $r$, and
then retaining only the type coordinates.  Successive use of the Poisson
Janossy formula shows that, up to a positive combinatorial factor depending on
$n$ and $\tau$ but not on $r$, this measure has the kernel-product
representation
\begin{equation}\label{eq:rooted-tree-weight}
 \pi(\dd x_r)
 \prod_{z\in[n]}e^{-K(x_z,S)}
 \prod_{(z,w)\in\vec E_r(\tau)}K(x_z,\dd x_w).
\end{equation}
Here the product means iterated integration from the root towards the leaves.
There is no degree factorial in \eqref{eq:rooted-tree-weight}: in the Janossy
measure the factor $1/k!$ accounts for the $k!$ permutations of a $k$-tuple,
while the distinct auxiliary marks give exactly one labelled assignment for
each permutation.  Equivalently, order each offspring configuration by its
auxiliary marks; on this ordered coordinate chart the two factors cancel.  This
derives the stated kernel-product representation as an equality of measures,
rather than as a formal density.  The common product of the auxiliary-mark Lebesgue measures
has been suppressed.

If $r$ and $r'$ are adjacent, the two directed edge sets differ only by reversing
the edge between them.  The corresponding factors in
\eqref{eq:rooted-tree-weight} are the measures
\(\pi(\dd x_r)K(x_r,\dd x_{r'})\) and
\(\pi(\dd x_{r'})K(x_{r'},\dd x_r)\), which are equal by detailed balance.
All remaining edge orientations and all vertex exponential factors are
unchanged.  Iterated integration in the kernel-product representation therefore
shows that $Q_{\tau,r}$ and $Q_{\tau,r'}$ agree on measurable rectangles in
$S^n$.  Uniqueness of finite measures on the product sigma-field then gives
$Q_{\tau,r}=Q_{\tau,r'}$.  Connectivity of $\tau$ gives
\begin{equation}\label{eq:all-root-measures-equal}
 Q_{\tau,1}=\cdots=Q_{\tau,n}.
\end{equation}

Sum \eqref{eq:all-root-measures-equal} over labelled tree shapes $\tau$ and over
$n$.  Since the marked-tree space is standard Borel, these equalities may be
disintegrated with respect to the unrooted typed and auxiliary-marked tree.
They say that, conditional on a realization with $n$ vertices, each of its
$n$ labels has the same conditional probability of being the root; the
probability is therefore $1/n$.  Forgetting the independent auxiliary marks
preserves uniformity over the actual vertices.  The assumption of almost-sure
finiteness ensures that the finite disintegration covers the whole law.
\end{proof}

\begin{corollary}[Negative first moment]\label{cor:negative-moment}
Under the assumptions of Lemma~\ref{lem:rerooting}, if
\[
 \bar m:=\int_SK(x,S)\pi(\dd x)<\infty,
\]
then
\begin{equation}\label{eq:negative-moment}
 \bE_\pi\left[\frac1T\right]=1-\frac{\bar m}{2}.
\end{equation}
\end{corollary}

\begin{proof}
The degree of the root is its number of offspring, so its expected degree is
$\bar m$.  Let $\mathcal U$ be the underlying unrooted typed tree.  Conditional
on $\mathcal U$, if it has $T$ vertices, Lemma~\ref{lem:rerooting} and the
handshaking lemma give
\[
 \bE[\deg(o)\mid\mathcal U]=\frac{2(T-1)}{T}
 =2\left(1-\frac1T\right).
\]
Taking expectations and rearranging proves \eqref{eq:negative-moment}.
\end{proof}

\begin{remark}
In the one-type case, \eqref{eq:negative-moment} also follows by summing the
Borel total-progeny law in equation~(10) of \cite{AldousPitman1998}; the
uniform-root symmetry conditional on total size is reflected in their
Proposition~1, equation~(15).
\end{remark}

\begin{remark}
The Poisson assumption is important in Lemma~\ref{lem:rerooting}.  For a general
offspring distribution, re-rooting changes the product of offspring
probabilities, and detailed balance of the mean kernel alone need not imply a
uniform root.
\end{remark}

\section{Component densities of sparse random graphs}\label{sec:graphs}

We first define the inhomogeneous random graph and its associated branching
process under the graphical-kernel hypotheses of Bollob\'as, Janson and Riordan
\cite{BJR2007}.  Their fixed-component asymptotics give the reciprocal-progeny
representation of the limiting component density.  We then combine the
re-rooting identity of Section~\ref{sec:rerooting} with branching-process
duality to evaluate this expectation.

\subsection{The BJR model and its branching process}

Let $(S,d)$ be a separable metric space, let $\pi$ be a Borel
probability measure, and, for each $n$, let
$\boldsymbol x_n=(x_1^{(n)},\ldots,x_n^{(n)})$ be a possibly random sequence of
types.  Write $\mathcal V=(S,\pi,(\boldsymbol x_n)_{n\geq1})$.  We call
$\mathcal V$ a \emph{vertex space} if its empirical type measure
\begin{equation}\label{eq:vertex-space}
 \nu_n:=\frac1n\sum_{i=1}^n\delta_{x_i^{(n)}}
 \xrightarrow{\bP}\pi
\end{equation}
weakly; equivalently, $\nu_n(A)\to\pi(A)$ in probability for every
$\pi$-continuity set $A$.  Let $\kappa:S^2\to[0,\infty)$ be symmetric and
measurable, and fix $u\geq0$.

For symmetric measurable kernels $\kappa_n:S^2\to[0,\infty)$, define
$G_n=G^{\mathcal V}(n,\kappa_n)$ conditionally on the types by taking distinct
edges independently with
\begin{equation}\label{eq:BJR-edge-prob}
 \bP(ij\in E(G_n)\mid\boldsymbol x_n)
 =p_{ij}^{(n)}
 :=\min\left\{\frac{\kappa_n(x_i^{(n)},x_j^{(n)})}{n},1\right\}.
\end{equation}
We say that $(\kappa_n)$ is a \emph{graphical sequence with limit $u\kappa$}
when the following three conditions hold:
\begin{enumerate}[label=\textup{(G\arabic*)},leftmargin=2.8em]
 \item $u\kappa$ is continuous $\pi\otimes\pi$-almost everywhere and belongs
       to $L^1(\pi\otimes\pi)$;
 \item for $\pi\otimes\pi$-almost every $(x,y)$, whenever $x_n\to x$ and
       $y_n\to y$, one has $\kappa_n(x_n,y_n)\to u\kappa(x,y)$;
 \item the expected edge density has the correct limit,
 \begin{equation}\label{eq:graphical-edge-assumption}
  \frac1n\bE|E(G_n)|\longrightarrow
  \frac12\iint_{S^2}u\kappa(x,y)\pi(\dd x)\pi(\dd y).
 \end{equation}
\end{enumerate}
These are precisely the vertex-space and graph construction on pp.~7--8 and
the graphical-sequence conditions in Definitions~2.7 and~2.9 of
\cite{BJR2007}, specialized to a probability ground space and limit
$u\kappa$.  For the constant sequence $\kappa_n=u\kappa$, condition
\eqref{eq:graphical-edge-assumption} says that $u\kappa$ is graphical.  In
particular, it is automatic for bounded, almost-everywhere continuous kernels
by Lemma~8.1 of \cite{BJR2007}.

The branching process associated with $u\kappa$ is the following typed Poisson
Galton--Watson tree $\cT_u$.  The root type has law $\pi$.  Independently for
different particles, a particle of type $x$ has a Poisson point process of child
types with intensity
\begin{equation}\label{eq:IRG-branching-kernel}
 K_u(x,\dd y):=u\kappa(x,y)\pi(\dd y).
\end{equation}
Since $u\kappa\in L^1(\pi\otimes\pi)$, the row integral is finite for
$\pi$-almost every $x$.  For the branching-process construction, choose a
finite symmetric representative as follows.  Let
\[
 N:=\left\{x:\int_Su\kappa(x,y)\pi(\dd y)=\infty\right\}
\]
and replace $\kappa(x,y)$ by zero whenever $x\in N$ or $y\in N$, retaining the
notation $\kappa$ for the modified representative.
The set $N$ is measurable and $\pi(N)=0$; the modification is symmetric, is
$\pi\otimes\pi$-null, and makes \eqref{eq:IRG-branching-kernel} finite at every
type without changing any $\pi$-rooted branching-process law, integral, or BJR
limit below.  The graphical sequence itself is not modified.
Write $\bP_{x,u}$ and $\bE_{x,u}$ when the root
type is fixed at $x$, and $\bP_\pi$ and $\bE_\pi$ when it has law $\pi$.
Write $T_u\in\mathbb N\cup\{\infty\}$ for the total progeny, and use the
convention $1/T_u=0$ when $T_u=\infty$.  Since every particle has finitely many
children almost surely, survival forever is equivalent to $T_u=\infty$.  Define
\begin{equation}\label{eq:gamma-definition}
 \gamma(u):=\bE_\pi\left[\frac1{T_u}\right]
 =\sum_{k\geq1}\frac{\bP_\pi(T_u=k)}k.
\end{equation}

For the BJR model, Theorem~9.1 of \cite{BJR2007} states that, for each fixed
$k\geq1$,
\begin{equation}\label{eq:BJR-Nk}
 \frac{N_k(G_n)}n\xrightarrow{\bP}\bP_\pi(T_u=k),
\end{equation}
where $N_k(G_n)$ is the number of vertices in components of order $k$.

\begin{theorem}[Component density in the BJR model]\label{thm:BJR-component}
Under \eqref{eq:vertex-space}, \eqref{eq:BJR-edge-prob}, and conditions
\textup{(G1)}--\textup{(G3)},
\begin{equation}\label{eq:BJR-component-limit}
 \frac{K(G_n)}n\longrightarrow\gamma(u)
 =\bE_\pi\left[\frac1{T_u}\right]
 \qquad\text{in probability and in $L^1$.}
\end{equation}
Consequently, $\bE K(G_n)/n\to\gamma(u)$.
\end{theorem}

\begin{proof}
Every component of order $k$ contains $k$ vertices, so
\begin{equation}\label{eq:component-size-sum}
 \frac{K(G_n)}n=\sum_{k\geq1}\frac{N_k(G_n)}{kn}.
\end{equation}
Fix $R\geq1$.  By \eqref{eq:BJR-Nk}, the finite sum in
\eqref{eq:component-size-sum} over $k\leq R$ converges in probability to
$\sum_{k\leq R}\bP_\pi(T_u=k)/k$.  The two omitted tails satisfy the
deterministic bounds
\begin{align*}
 \sum_{k>R}\frac{N_k(G_n)}{kn}
 &\leq\frac1{R+1}\sum_{k>R}\frac{N_k(G_n)}n
 \leq\frac1{R+1},\\
 \sum_{k>R}\frac{\bP_\pi(T_u=k)}k
 &\leq\frac1{R+1}\sum_{k>R}\bP_\pi(T_u=k)
 \leq\frac1{R+1}.
\end{align*}
Given $\varepsilon>0$, first take $R$ so large that $2/(R+1)<\varepsilon/2$,
and then use convergence of the finite sum to see that the probability that the
absolute difference in \eqref{eq:BJR-component-limit} exceeds $\varepsilon$
tends to zero.  Finally $0\leq K(G_n)/n\leq1$, so convergence in probability to
the constant $\gamma(u)$ implies $L^1$ convergence and convergence of
expectations.
\end{proof}

\subsection{Explicit subcritical and supercritical limits}

Throughout this subsection assume $u\kappa\in L^1(\pi\otimes\pi)$.
Let
\begin{equation}\label{eq:q-definition}
 q_u(x):=\bP_{x,u}(T_u<\infty)
\end{equation}
be the extinction probability when the root type is fixed at $x$.
Lemma~5.6(iv) and Theorem~6.2 of \cite{BJR2007} state that $1-q_u$ is the
maximal nonnegative solution of the survival equation, up to $\pi$-null sets.
For the finite representative fixed above, the branching property gives, for
every $x\in S$,
\begin{equation}\label{eq:q-fixed-point}
 q_u(x)=\exp\left\{-\int_Su\kappa(x,y)(1-q_u(y))\pi(\dd y)\right\}.
\end{equation}
Only the $\pi$-almost-everywhere equivalence class of this solution enters the
BJR graph limit.  Define
\begin{equation}\label{eq:c-measure}
 c_u(\dd x):=q_u(x)\pi(\dd x),\qquad C_u:=c_u(S).
\end{equation}

In the supercritical case, BJR duality (Definition~3.15 and the paragraph
immediately preceding Lemma~6.6 of \cite{BJR2007}) states that, conditional on
extinction and on root type $x$,
$\cT_u$ is a Poisson Galton--Watson tree $\cT_u^q$ with offspring kernel
\begin{equation}\label{eq:dual-kernel}
 K_u^q(x,\dd y):=u\kappa(x,y)c_u(\dd y).
\end{equation}
To see the thinning mechanism, condition on the offspring types of a particle
of type $x$.  The descendant trees become extinct independently, with
probabilities $q_u(y)$ for children of type $y$.  Weighting the offspring
Poisson process by the product of these probabilities and using its generating
functional gives the normalizing constant $q_u(x)$ by
\eqref{eq:q-fixed-point} and changes its intensity from
$u\kappa(x,y)\pi(\dd y)$ to
$u\kappa(x,y)q_u(y)\pi(\dd y)=K_u^q(x,\dd y)$; conditional descendant trees
remain independent.  Iteration gives the quoted dual process.
When extinction is almost sure, the same assertion is immediate because
$q_u=1$ almost everywhere and \eqref{eq:dual-kernel} is the original kernel.
Write $\bE_x^q$ for expectation under this dual process with root type $x$.
For a probability measure $\eta$ on $S$, write
$\bE_\eta^q:=\int_S\bE_x^q\eta(\dd x)$.
Thus, if $\bar c_u:=c_u/C_u$ and $F$ is any nonnegative measurable functional
of a finite typed tree, conditioning first on the root type gives
\begin{equation}\label{eq:dual-functional}
 \bE_\pi[F(\cT_u);T_u<\infty]
 =\int_S c_u(\dd x)\,\bE_x^qF(\cT_u^q)
 =C_u\bE_{\bar c_u}^qF(\cT_u^q).
\end{equation}
The row intensity in \eqref{eq:IRG-branching-kernel} is finite for every $x$
under the chosen representative, so the zero-offspring event gives $q_u(x)>0$
for every $x$; hence $C_u>0$ and $\bar c_u$ is well defined.

\begin{theorem}[General component-density formula]\label{thm:IRG-general}
Under the branching-process setup of this section, suppose
$u\kappa\in L^1(\pi\otimes\pi)$.  Then
\begin{equation}\label{eq:general-density}
 \gamma(u)=C_u-\frac12
 \iint_{S^2}u\kappa(x,y)c_u(\dd x)c_u(\dd y).
\end{equation}
This is an identity for the associated branching process.  For the BJR model,
the integrability assumption is already part of \textup{(G1)}; if $G_n$
satisfies \textup{(G1)}--\textup{(G3)}, then
\begin{equation}\label{eq:general-graph-limit}
 \frac{K(G_n)}n\longrightarrow
 C_u-\frac12\iint_{S^2}u\kappa(x,y)c_u(\dd x)c_u(\dd y)
 \quad\text{in probability and in $L^1$},
\end{equation}
and the same limit holds for expectations.
\end{theorem}

\begin{proof}
By symmetry of $\kappa$ and \eqref{eq:dual-kernel},
\begin{equation}\label{eq:dual-detailed-balance}
 \bar c_u(\dd x)K_u^q(x,\dd y)
 =\frac1{C_u}u\kappa(x,y)c_u(\dd x)c_u(\dd y)
\end{equation}
is symmetric.  The dual tree is finite almost surely because it is the
conditional law of $\cT_u$ given extinction.  Moreover, since $0\leq q_u\leq1$,
\[
 \iint_{S^2}u\kappa(x,y)c_u(\dd x)c_u(\dd y)
 \leq\iint_{S^2}u\kappa(x,y)\pi(\dd x)\pi(\dd y)<\infty.
\]
Thus its mean root degree is finite and equals
\[
 \int_S K_u^q(x,S)\bar c_u(\dd x)
 =\frac1{C_u}\iint_{S^2}u\kappa(x,y)c_u(\dd x)c_u(\dd y).
\]
Corollary~\ref{cor:negative-moment}, applied with root law $\bar c_u$ and
kernel $K_u^q$, therefore gives
\begin{equation}\label{eq:dual-negative-moment}
 \bE_{\bar c_u}^q\frac1{T_u^q}
 =1-\frac1{2C_u}
   \iint_{S^2}u\kappa(x,y)c_u(\dd x)c_u(\dd y).
\end{equation}

Because $1/T_u=0$ on survival, equation \eqref{eq:dual-functional} with
$F=1/T$ yields
\[
 \gamma(u)=\bE_\pi\frac1{T_u}
 =C_u\bE_{\bar c_u}^q\frac1{T_u^q}.
\]
Substitution of \eqref{eq:dual-negative-moment} proves
\eqref{eq:general-density}.  The graph limits now follow from
Theorem~\ref{thm:BJR-component}; its $L^1$ assertion also gives convergence of
expectations.
\end{proof}

\begin{remark}[Comparison with component-measure formulas]
Andreis, K\"onig, Langhammer and Patterson assume that the type space is compact
metric and that the kernel is nonnegative, continuous and irreducible; see the
first paragraph of Section~1.1 in \cite{AndreisEtAl2023}.  Their Theorem~2.1 and
equation~(2.5) identify the law of large numbers for the empirical microscopic
component measure.  Equation~(2.12), together with Remark~4.6 and Lemma~4.7,
directly relates this measure to the total-progeny law of the associated
branching process.  Lemmas~6.7(2) and~6.8, especially equation~(6.16), then give
its total mass as $c(S)-\langle c,\kappa c\rangle/2$, including the critical
case.  Yu and Sun work with a finite type space; Corollary~1 of
\cite{YuSun2024} states the same component-count limit outside criticality, and
Remark~1 states the reciprocal-total-progeny representation and suggests a
proof based on pruning.

For this law-of-large-numbers conclusion, Theorem~\ref{thm:IRG-general} applies
under the broader BJR assumptions: the type space need only be separable, and a
graphical limit kernel may be almost-everywhere continuous, integrable,
unbounded, and reducible.  The proof is also different, using Theorem~9.1 of
\cite{BJR2007}, the reversible-tree identity, and extinction duality.  This
comparison is limited to the component-density law of large numbers: the
large-deviation principle of \cite{AndreisEtAl2023} and the moderate-deviation
principles of \cite{YuSun2024} are stronger results in other directions.
\end{remark}

\begin{corollary}[Subcritical and critical component density]
\label{thm:IRG-subcritical}
Suppose $u\kappa\in L^1(\pi\otimes\pi)$ and $\cT_u$ becomes extinct almost
surely.  Then $q_u=1$ $\pi$-almost everywhere, $c_u=\pi$, and
\begin{equation}\label{eq:finite-average-degree}
 m_\pi(u):=\iint_{S^2}u\kappa(x,y)\pi(\dd x)\pi(\dd y)<\infty.
\end{equation}
In this case
\begin{equation}\label{eq:subcritical-density}
 \gamma(u)=\bE_\pi\frac1{T_u}=1-\frac{m_\pi(u)}2,
\end{equation}
and the BJR graph in Theorem~\ref{thm:IRG-general} satisfies
\begin{equation}\label{eq:IRG-subcritical-limit}
 \frac{K(G_n)}n\longrightarrow1-\frac{m_\pi(u)}2
 \qquad\text{in probability and in $L^1$.}
\end{equation}
\end{corollary}

\begin{proof}
Under almost-sure extinction, the definitions give $q_u=1$, $C_u=1$, and
$c_u=\pi$.  The integrability assumption in
Theorem~\ref{thm:IRG-general} makes \eqref{eq:finite-average-degree} finite.
The assertions are now the specialization of
\eqref{eq:general-density} and \eqref{eq:general-graph-limit}.
\end{proof}

For a checkable extinction criterion, define the symmetric,
positivity-preserving integral operator by
\[
 (\mathsf T_{u\kappa}f)(x)
 :=\int_Su\kappa(x,y)f(y)\pi(\dd y).
\]
Following equation~(2.15) of \cite{BJR2007}, its extended $L^2(\pi)$ norm is
\[
 \|\mathsf T_{u\kappa}\|_{2\to2}
 :=\sup\{\|\mathsf T_{u\kappa}f\|_2:f\geq0,\ \|f\|_2\leq1\}\in[0,\infty].
\]
Theorem~6.2 of \cite{BJR2007} states that
$\|\mathsf T_{u\kappa}\|_{2\to2}\leq1$ implies extinction for
$\pi$-almost every initial type (and that positive survival has positive
$\pi$-measure exactly when the norm exceeds one).  Thus the preceding
corollary applies throughout the subcritical regime and at the operator
critical point; irreducibility is not required for this implication.

\begin{remark}[Finite-type form]
When $S=\{1,\ldots,d\}$, $\pi=(\pi_i)$, and $c_i=\pi_iq_u(i)$,
\eqref{eq:general-density} becomes
\[
 \gamma(u)=|c|-\frac u2c^{\mathsf T}\kappa c.
\]
Here $|c|:=\sum_{i=1}^d c_i$ and $\kappa$ denotes the matrix
$(\kappa(i,j))_{i,j\leq d}$.
For $u=1$, this is the law-of-large-numbers value in Corollary~1 of
\cite{YuSun2024}.  Remark~1 of that paper states the
reciprocal-total-progeny representation and suggests a proof based on pruning.
Equations \eqref{eq:gamma-definition} and \eqref{eq:general-density} derive the
representation from fixed-component asymptotics, re-rooting, and extinction
duality under the more general BJR hypotheses.  For compact continuous kernels,
equation~(2.12) of \cite{AndreisEtAl2023} already gives the branching-process
connection, while equation~(6.16) in Lemma~6.7(2), extended to criticality by
Lemma~6.8, gives the same total-mass identity.
\end{remark}

\begin{remark}[Euler characteristic]
For every finite graph,
\[
 K(G)=|V(G)|-|E(G)|+\operatorname{sur}(G),
\]
where
$\operatorname{sur}(G):=\sum_C(|E(C)|-|V(C)|+1)$ is the total cycle surplus,
the sum being over connected components.  Proposition~8.9 of
\cite{BJR2007} gives directly, under \textup{(G1)}--\textup{(G3)},
\[
 \frac{|E(G_n)|}{n}\xrightarrow{\bP}\frac{m_\pi(u)}2.
\]
In the extinction regime, combining this quoted result with
\eqref{eq:IRG-subcritical-limit} and the Euler identity yields
$\operatorname{sur}(G_n)/n\to0$ in probability.
\end{remark}

\section*{Acknowledgments}

The authors thank Hui He of Beijing Normal University for helpful discussions
during the early stages of this work.  This work is supported by the National
Key R\&D Program of China (No.~2022YFA1006500) and by the National Natural
Science Foundation of China (No.~12401171).

\small
\bibliographystyle{abbrv}
\bibliography{pruning-revised}

\end{document}